\documentclass[10pt, a4paper]{amsart}
\usepackage{amssymb, amsmath, amsfonts, amsthm, verbatim}
\usepackage{amsaddr}
\usepackage[numbers]{natbib}
\usepackage{fullpage}
\usepackage{color}

\newtheorem{definition}[equation]{Definition}
\newtheorem{lemma}[equation]{Lemma}
\newtheorem{theorem}[equation]{Theorem}
\newtheorem{cor}[equation]{Corollary}
\newtheorem{remark}[equation]{Remark}

\theoremstyle{remark}

\newcommand{\FF}{\mathbb{F}}
\newcommand{\ZZ}{\mathbb{Z}}

\newcommand{\C}{\mathcal{C}} 
\newcommand{\M}{M\"obius }
\newcommand{\sg}{\leqslant}
\newcommand{\nsg}{\trianglelefteq}

\newcommand{\J}{t}
\newcommand{\T}{u}

\newcommand{\R}{\textnormal{\textbf{\emph{R}}}}
\newcommand{\para}{\textnormal{\textbf{\emph{P}}}}
\newcommand{\norm}{\textnormal{\textbf{\emph{N}}}\!}
\newcommand{\cent}{\textnormal{\textbf{\emph{C}}}}
\newcommand{\dee}{\textnormal{\textbf{\emph{D}}}}

\newcommand{\ee}{\textnormal{\textbf{\emph{E}}}}
\newcommand{\eff}{\textnormal{\textbf{\emph{F}}}}
\newcommand{\vier}{\textnormal{\textbf{\emph{V}}}\!}
\newcommand{\eye}{\textnormal{\textbf{\emph{I}}}}
\newcommand{\mi}{\textnormal{\textbf{MaxInt}}}

\title{The M\"obius function of the small Ree groups}
\author{Emilio Pierro}
\date{}
\address{Department of Economics, Mathematics and Statistics, \\ Birkbeck, University of London, \\ Malet Street, London, WC1E 7HX}
\email{e.pierro@mail.bbk.ac.uk}

\begin{document}

\begin{abstract} The \M function for a group, $G$, was introduced in 1936 by Hall in order to count ordered generating sets of $G$. In this paper we determine the \M function of the simple small Ree groups, $R(q)={}^2G_2(q)$ where $q=3^{2m+1}$ for $m>0$, using their 2-transitive permutation representation of degree $q^3+1$ and describe their maximal subgroups in terms of this representation. We then use this to determine $\vert$Epi$(\Gamma,G)\vert$ for various $\Gamma$, such as $F_2$ or the modular group $PSL_2(\ZZ)$, with applications to Grothendieck's theory of dessins d'enfants as well as probabilistic generation of the small Ree groups.\end{abstract}
\maketitle
\section{Introduction}
The \M function of a finite group, $G$, was introduced by Hall in 1936 \cite{hall} in order to count the number of ways $G$ can be generated by an ordered subset of its elements. Denote by $\sigma_n(G)$ the number of ordered subsets $\{x_1,\dots,x_n\}$ of $G$ and by $\phi_n(G)$ the number of these subsets which in addition generate $G$. Since any subset of $G$ generates a subgroup $H \sg G$ we have the immediate identity \[\sigma(G)=\sum_{H \sg G} \phi(H)\] and Hall observed that \M inversion could be applied to this relation to give \[ \phi(G)=\sum_{H \sg G}\mu_G(H)\sigma(H)\] where the \M function of a subgroup $H$ is determined by \[\sum_{H \sg K}\mu_G(K)=\begin{cases}
1	&\text{if }H = G\\
0	&\text{if }H \neq G.
\end{cases}\]

A priori, it seems as though we might have to work through the entire subgroup lattice of $G$, but since it is clear that $\mu_G(H)=\mu_G(H')$ if $H$ and $H'$ are conjugate in $G$, then we need only determine $\mu_G(H)$ on a set of conjugacy class representatives of subgroups. In fact, due to the following theorem of Hall, we only need to determine $\mu_G(H)$ for a set of conjugacy class representatives of subgroups which occur as the intersection of maximal subgroups and the size of those conjugacy classes. \begin{theorem}[Hall] \label{hall} If $H \sg G$ then $\mu_G(H)=0$ unless $H=G$ or $H$ is the intersection of maximal subgroups of $G$. \end{theorem}

The related problem of subjecting our subsets to some relations can also be solved. Let $f_j(x_i)$ be a family of words in the $x_i$ and $\Gamma$ the finitely-generated abstract group generated by these words. We modify our notation to let $\sigma_\Gamma(G)$ count the ordered subsets of $G$ of size $n$ which satisfy the relations $f_j$ and $\phi_\Gamma(G)$ those which additionally generate. The quantity $\phi_\Gamma(G)$ then counts the number of epimorphisms $\theta \colon \Gamma \to G$ and since Aut$(G)$ acts semiregularly \cite{hall} on the ordered $n$-subsets of $G$ it follows that $d_\Gamma=\phi_\Gamma(G)/\vert$Aut$(G)\vert$ counts the number of distinct normal subgroups $N \nsg \Gamma$ such that $\Gamma/N \cong G$. If the elements of a group $\Gamma$ correspond to some category of objects, then the quantity $d_\Gamma$ enumerates the number of those objects whose automorphism group is isomorphic to $G$, see \cite{some,dj87,suzenum,suzmob}. For example, the quantity $d_2(G)=d_{F_2}(G)$, where $F_2$, the free group on 2 generators, is connected to Grothendieck's theory of \emph{dessins d'enfants} \cite{dd} where it is equal to the number of regular dessins with automorphism group $G$.

Hall succeeds in determining the \M function for the family of simple groups $PSL_2(p)=L_2(p)$ where $p>3$ is prime as well as a number of other finite group. In 1988 these results were extended by Downs in his PhD thesis \cite{phdowns} to $L_2(q)$ and $PGL_2(q)$ for all prime powers, $q$. The determination of $L_2(q)$ was published in \cite{psl2} and applications of this function were published in \cite{some, dj87}. More recently Downs and Jones have determined the \M function of the simple Suzuki groups, $Sz(2^{2m+1}) \cong {}^2B_2(2^{2m+1})$ \cite{suzmob} and have applied these results to enumeration in various Hecke groups, in addition to those objects already mentioned \cite{suzenum}.

Following the Suzuki groups it seems natural to determine the \M function for the small Ree groups, $R(q)$, first constructed by Ree in 1961 \cite{ree}. Like the Suzuki groups, the small Ree groups are one of the few families of twisted Chevalley groups. They are usually considered as a group of $7 \times 7$ matrices in $L_7(q)$, for which explicit generators are given \cite{ln}, or as a group of $8 \times 8$ matrices in $P\Omega_8^+(q)$ as in \cite{kleid}. In this paper we will mostly work with the 2-transitive permutation representation of degree $q^3+1$ due to Tits \cite{tits}, however this still depends on the Lie theory for its construction. More recently, Wilson \cite{raw1, raw2, raw12} has determined a construction of the small Ree groups without any use of the Lie theory. The small Ree groups, like the Suzuki groups, also have an interpretation in terms of a finite geometry, see \cite[Section 7.7]{van} and a design, in the case of $R(q)$, a $2-(q^3+1,q+1,1)$ design \cite{lun}. The action of elements of $G$ on $\Omega$ is well known \cite{ward}, in addition we describe the action of the maximal subgroups of $G$ on $\Omega$. The author makes no claim that this interpretation is not known elsewhere but is not aware of it appearing in the literature so we state these findings since they may be of interest in their own right and since they will be used in the proofs later on. The main aim of this paper is to prove the following.

\begin{theorem} \label{mob} Let $G=R(3^n)$, for a positive odd integer $n>3$, be a simple small Ree group. Then the \M function, $\mu_G(H)$ of a subgroup $H \sg G$ is zero except for the following families of conjugacy classes of subgroups of $G$.
\begin{center}
\begin{tabular}{c c c c}\hline
Isomorphism								&for $h \vert n$	&											&\\
type of $H \sg G$							&and s.t.		&$[G\colon N_G(H)]$						&$\mu_G(H)$\\ \hline
$R(3^h)$									&--				&$\vert G \vert/3^{3h}(3^{3h}+1)(3^h-1)$	&$\mu(n/h)$\\
$3^h+\sqrt{3^{h+1}}+1\colon6$				&--				&$\vert G \vert/6(3^h+\sqrt{3^{h+1}}+1)$	&$-\mu(n/h)$\\
$3^h-\sqrt{3^{h+1}}+1\colon6$				&$h>1$			&$\vert G \vert/6(3^h-\sqrt{3^{h+1}}+1)$	&$-\mu(n/h)$\\
$(3^h)^{1+1+1}\colon(3^h-1)$				&--				&$\vert G \vert/3^{3h}(3^h-1)$				&$-\mu(n/h)$\\
$2 \times L_2(3^h)$							&$h>1$			&$\vert G \vert/3^h(3^{2h}-1)$				&$-\mu(n/h)$\\
$2 \times (3^{h}\colon\frac{3^h-1}{2})$			&$h>1$			&$\vert G \vert/3^h(3^h-1)$					&$\mu(n/h)$\\
$(2^2 \times D_{(3^h+1)/2})\colon3$			&$h>1$			&$\vert G \vert/6(3^h+1)$					&$-\mu(n/h)$\\
$2^2 \times D_{(3^h+1)/2}$					&$h>1$			&$\vert G \vert/6(3^h+1)$					&$3\mu(n/h)$\\
$2 \times L_2(3)$							&--				&$\vert G \vert/24$							&$-2\mu(n)$\\
$2^3$										&--				&$\vert G \vert/168$						&$21\mu(n)$\\ \hline
\end{tabular}
\end{center}
where $\mu(n)$ for a natural number, $n$, is the usual \M function from number theory. From the number of subgroups conjugate to $H$ in $G$, given by $[G:N_G(H)]$ the index of the normaliser of $H$ in $G$, the \M function of $G$ can be determined.
\end{theorem}

Following the determination of the \M function for the simple small Ree groups we use it to determine $d_\Gamma(G)$ for a number of finitely presented groups and use these to prove a number of related results in probabilistic generation of the small Ree groups. Many of these results were previously known, see \cite{lieb} for a recent survey in this area, an exception is the following which the author is unaware of appearing elsewhere in the literature.

\begin{cor} \label{23ree} Let $G=R(q)$ be a small Ree group and $P_{a,b}(G)$ be the probability that a randomly chosen element of order $a$ and a randomly chosen element of order $b$ will generate $G$. Then $P_{2,3}(G) \to 1$ as $\vert G \vert \to \infty$ and $P_{3,3}(G) \to 1$ as $\vert G \vert \to \infty$.
\end{cor}

This paper is then structured as follows. In Section \ref{struct} we give an analysis of the structure of $R(q)$, its maximal subgroups and their behaviour as permutation groups which we use in Section \ref{maxint} to prove which subgroups of $R(q)$ occur as an intersection of maximal subgroups. In Section \ref{det} we then determine the \M function of $R(q)$ and in Section \ref{app} we evaluate the \M function for a number of finitely generates groups, such as $F_2$ and $PSL_2(\ZZ)$ and use this to prove a number of results regarding the probabilistic generation of the small Ree groups.

The smallest small Ree group $R(3) \cong P \Gamma L_2(8)$ is not simple, but its derived subgroup $R(3)' \cong L_2(8)$ is and the \M function for $R(3)'$ was determined by Downs in \cite{psl2}. However, since the maximal subgroups of $R(3)$ differ from those of the simple small Ree groups we do not consider it in this paper. We also make use of the ATLAS \cite{ATLAS} notation throughout.

%%%%%%%%%%%%%%%%%%%%%%%%%%%---Structure---%%%%%%%%%%%%%%%%%%%%%%%%%%%
\section{Structure of the small Ree groups} \label{struct}
Throughout, $G$ will denote a simple small Ree group, $R(q)$, where $q>3$ is an odd power of 3, and $\Omega$ will be a set of size $q^3+1$ on which $G$ acts 2-transitively, as in \cite{tits}. In order to describe the maximal subgroups in detail, we first describe the conjugacy classes of elements in $R(q)$ and in particular the action of their elements on $\Omega$.
\subsection{Conjugacy Classes}
The following assembles the necessary results from the character table of $R(q)$, due to Ward \cite{ward}, as well as results from Levchuk and Nuzhin \cite{ln} and the summary given by Jones in \cite{rr}.

The Sylow 2-subgroups of $G$ are elementary abelian of order 8, so $G$ contains no elements of order 4, and the normaliser in $G$ of a Sylow 2-subgroup has shape $2^3\colon7\colon3$. All involutions inside a Sylow 2-subgroup normaliser are conjugate, hence they are all conjugate in $G$. An involution in $G$ is represented by $\J$ and will fix $q+1$ points in $\Omega$, which we denote $\Omega^\J$ or refer to as the \emph{block} of $\J$, and the centraliser in $G$ of $\J$ is $C_G(\J) \cong 2 \times L_2(q)$, which acts doubly transitively on the block of $\J$ \cite{lun}. Any two blocks can intersect in at most 1 point and any two points are pointwise fixed by a unique involution.

The Sylow 3-subgroups of $G$ have shape $q^{1+1+1}$, exponent 9 and trivial intersection with each other.
They each stabilise a unique point $\omega \in \Omega$.
Elements of order 3 fix a single point in $\Omega$ and fall into one of 3 conjugacy classes, $\C_3^0$, $\C_3^+$ and $\C_3^-$. Elements of $\C_3^0$ have centralisers of order $q^3$ and are conjugate to their inverses, while elements of $\C_3^*=\C_3^+ \cup \C_3^-$, denoted $\T$, have centralisers of order $2q^2$ and if $\T \in \C_3^+$ then $\T^{-1} \in \C_3^-$. 
Elements of order 6 fix a single point in $\Omega$ and fall into two conjugacy classes, $\C_6^+$, $\C_6^-$. They are denoted by $\J\T$, the commuting product of an involution with an element of $\C_3^*$, hence they square to elements of $\C_3^*$. Their centralisers have order $2q$ and if $\J\T \in \C_6^+$ then $\J\T^{-1} \in \C_6^-$.
Elements of order 9 also fix a single point in $\Omega$ and fall into conjugacy classes $\C_9^0$, $\C_9^+$ and $\C_9^-$ where elements in $\C_9^0$ are conjugate to their inverses but if $g \in \C_9^+$ then $g^{-1} \in \C_9^+$. Elements of order 9 all have centralisers in $G$ of order $3q$ and will cube to an element of $\C_3^0$.

The remaining elements are best described in terms of the Hall subgroups to which they belong. The Hall subgroups $A_i$, for $i=0,1,2,3$, are each cyclic, having the following orders \[\vert A_0 \vert = \frac{1}{2}(q-1), \; \; \vert A_1 \vert = \frac{1}{4}(q+1), \; \; \vert A_2 \vert = q-\sqrt{3q}+1, \; \; \vert A_3 \vert = q+\sqrt{3q}+1\] so that $G$ has order \[ \vert G \vert = 2^3q^3\vert A_0 \vert \vert A_1 \vert \vert A_2 \vert \vert A_3 \vert.\]
Elements of order $k \neq 2$ dividing $q-1$ are conjugate to some power of $\J r$, the commuting product of an involution, $\J$, with an element, $r$, belonging to a Hall subgroup $A_0$. Such an element fixes two points in $\Omega$ and $\vert C_G(\J r) \vert = q-1$. The remaining elements do not fix any points in $\Omega$ and have order $k \geq 7$ dividing $q^3+1 = (q+1)(q-\sqrt{3q}+1)(q+ \sqrt{3q}+1)$. Elements of order 7 are all conjugate and will divide one of these three factors of $q^3+1$ according to the value of $q$ modulo 7. Otherwise there are elements of order $k \neq 2$ dividing $(q+1)/2$ conjugate to some power of $\J s$, the commuting product of an involution, $\J$, with an element, $s$ of a Hall subgroup, $A_1$; elements, $v$, of order dividing $q-\sqrt{3q}+1$; and elements, $w$, of order dividing $q+ \sqrt{3q}+1$. Their centralisers in $G$ have the following orders: $\vert C_G(\J s) \vert = q+1$, $\vert C_G(v) \vert = q-\sqrt{3q}+1$ and $\vert C_G(w)\vert = q+ \sqrt{3q}+1$.

From Sylow's Theorems and the fact that a Hall subgroup is the centraliser of any Sylow $p$-subgroup which it contains, we have that if a subgroup is isomorphic to a Sylow 2-subgroup, Sylow 3-Subgroup, or Hall subgroup $A_i$, for $i=0,1,2,3$, then it is also conjugate to it.

\subsection{Maximal Subgroups}
The maximal subgroups for the simple small Ree groups were determined by Levchuk and Nuzhin \cite{ln} and Kleidman \cite{kleid}. They are conjugate to one of the following. \begin{enumerate}
\item subfield subgroups $R(q^{1/p})$ with $p$ prime,
\item parabolic subgroups, $q^{1+1+1}\colon(q-1)$,
\item involution centralisers, $2 \times L_2(q)$,
\item normalisers of a four-group, $(2^2 \times D_{(q+1)/2})\colon3$,
\item normalisers of a Hall subgroup $A_2$, $(q-\sqrt{3q}+1)\colon6$,
\item normalisers of a Hall subgroup $A_3$, $(q+\sqrt{3q}+1)\colon6$.
\end{enumerate}

%Subfield Subgroups
\subsubsection{Subfield subgroups}
If we think of $\Omega$ as a set of points in 6-dimensional projective space defined over $\FF_q$ \cite{tits}, or as a matrix group defined as a subgroup of $L_7(q)$ \cite{ln}, then a subfield subgroup is conjugate to the subgroup of $R(q)$ which fixes those points defined on the restriction to a subfield of $\FF_q$, equivalently, generated by those matrices whose entries belong to this subfield. These subgroups are maximal when the subfield is maximal in $\FF_q$, that is, when the index of the subfield is prime. If $q=3^n$, we write $G_m$ for a subgroup conjugate to the subfield subgroup defined over the field of order $3^m$ where $m \vert n$. Then $G_m$ will act doubly transitively on the subset $\Omega(m) \subset \Omega$ of size $3^{3m}+1$ which are the points in $\Omega$ fixed by the Sylow 3-subgroups of $G_m$. Points fixed by elements of $G_m$ all belong to the $\Omega(m)$ with the exception of the blocks of the involutions in $G_m$ which fix $3^m+1$ points in $\Omega(m)$ and $3^n-3^m$ points in $\Omega(m)^c$.

%Parabolic subgroups
\subsubsection{Parabolic subgroups}
The parabolic subgroups of $G$ are the normalisers in $G$ of a Sylow 3-subgroup and have shape $q^{1+1+1}\colon(q-1)$. They are stabilisers of a point $\omega \in \Omega$ and we denote them $P_\omega$ although we may suppress the $\omega$ if it is clear or unneccessary. The centre of a Sylow 3-subgroup is elementary abelian of order $q$ and all nontrivial elements belong to $\C_3^0$. It is contained in a normal elementary abelian subgroup of $P$ having order $q^2$ all of whose remaining elements belong to $\C_3^*$. All elements outside of this normal abelian subgroup have order 9. For a pair of involutions $\J \neq \J'$ in $P_\omega$ the intersection of their blocks is $\{\omega\}$.

%Involution centralisers
\subsubsection{Involution centralisers}
Let $\J \in G$ be an involution, $C=C_G(\J)$ its centraliser in $G$ and $\Omega^\J$ the block of $\J$ stabilised by $C$ and on which $C$ acts 2-transitively. Elements of $C$ of order 3 belong to $\C_3^*$ and fix a point in $\Omega^\J$, elements of order $k>2$ dividing $q-1$ fix two points in $\Omega^\J$ and elements of order dividing $(q+1)/2$ do not fix any points in $\Omega^\J$. It follows that any pair of commuting involutions in $G$ have disjoint blocks. If $\J' \neq \J''$ are non-commuting involutions in $C$ then from the list of maximal subgroups of $L_2(q)$ \cite{dick} we can easily determine the maximal subgroups of $C$ and we see that $\langle \J',\J'' \rangle$ generates a subgroup of a dihedral group. It follows that $\J'$ and $\J''$ both normalise a subgroup generated by an element of order $k>2$ where $k$ divides $(q-1)/2$ or $(q+1)/4$ and since such subgroups do not stabilise points outside of $\Omega^\J$ we have that $\Omega^{\J'} \cap \Omega^{\J''} = \varnothing$. There can then be at most $(q^3+1)/(q+1)=q^2-q+1$ involutions in $C$.

The involutions of $L = C/\langle\J\rangle \cong L_2(q)$ are all conjugate and so there are three conjugacy classes of involutions in $C$: \begin{enumerate}
\item \{t\}, the central involution,
\item the $q(q-1)/2$ involutions in $L$, and
\item the $q(q-1)/2$ involutions in the coset $\J L$
\end{enumerate}
so that there are a total of $q^2-q+1$ involutions in $C$ and we see that for any $\omega \in \Omega$, $\omega$ belongs to the block of some involution in $C$.

%Four-group normalisers
\subsubsection{Four-group normalisers}
Let $\J_1 \neq \J_2$ be commuting involutions in $G$ with $t_3=t_1t_2$. The four-group $V = \langle \J_1, \J_2 \rangle$ is centralised in $G$ by a dihedral subgroup $D_{(q+1)/2}$ and normalised by an element $\T \in \C_3^*$ so that we have $N=N_G(V)\cong(2^2 \times D_{(q+1)/2})\colon3$. This is also the normaliser of a Hall subgroup $\langle s \rangle$ conjugate to $A_1$ as follows. The centraliser of $\langle s \rangle$ in $G$ is a unique four-group, $V$, and $V \times \langle s \rangle$ is normalised by an element, $\tau\T$, of order 6, where $\tau$ commutes with $V$ and $\langle s \rangle^\T = \langle s \rangle$. A simple counting argument shows that $\langle s \rangle$ will belong to a unique four-group normaliser, whereas a four-group will belong to $1+3(q+1)/2$ four-group normalisers. To avoid confusion with the normalisers of the other Hall subgroups, we refer to groups conjugate to $N$ as four-group normalisers. There are four subgroups of $N$ isomorphic to $D_{(q+1)/2}$, three of which are conjugate and a representative of this conjugacy class we denote by $D'$, and one of which is normal, $D$, so that $\langle s \rangle \nsg D$. The three conjugacy classes of involutions in $N_G(V)$ are then \begin{enumerate}
\item the 3 involutions in $V$,
\item the $(q+1)/4$ involutions, $\tau$, in $D$, and;
\item the $3(q+1)/4$ involutions, $\tau'$, in $D'$, or one of its conjugates in $N$. \end{enumerate}
A four-group $V' \neq V$ in $N$ contains at most one involution from $V$ and at most one involution from $D$, otherwise $V'$ contains two distinct involutions conjugate to $\tau'$. Let $\tau_1'=\J_i\tau_1,\tau_2'=\J_j\tau_2$ be such a pair with $1 \leq i,j \leq 3$ and $\tau_1,\tau_2$ conjugate to $\tau$. Then their product $\J_i\tau_1\J_j\tau_2 = \J_i\J_js^n$ for some power of $s$ is an involution if and only if $\tau_1=\tau_2$. Thus, a four-group $V' \neq V$ contained in $N$ will be conjugate to either $\langle \J_i,\J_i\tau \rangle$ or $\langle \J_i,\J_j\tau \rangle$. The centraliser of $\tau$ in $N$ is $C_N(\tau) \cong 2 \times L_2(3)$; conversely, if $L$ is a subgroup of $N$ isomorphic to $2 \times L_2(3)$, then its central involution is conjugate to $\tau$ since the only conjugacy class of involutions whose order is not divisible 3 is that of $\tau$. The centraliser of $\tau'$ in $N$ is elementary abelian of order 8.

The geometric interpretation of $N$ is then as follows. The nontrivial elements of $V$ fix a mutually disjoint triple of blocks in $\Omega$. In fact, any pair of involutions in $N$ will fix a disjoint pair of blocks in $\Omega$ since all involutions belong to the centraliser of a four-group. The orbits of $s$ stabilise these blocks and since elements of order $(q+1)/2$ are conjugate to the commuting product of a nontrivial element of $V$ with an element conjugate to $\tau'$ we have that elements of order $(q+1)/2$ will be composed of four cycles of length $(q+1)/4$ and $2(q^2-q)$ cycles of length $(q+1)/2$. Elements conjugate to $\T$ or $\T^{-1}$ fix a point disjoint from the fixed points of all involutions, except those conjugate to $\tau$ and, from above, we have that $C_N(\tau)$ acts doubly transitively on four of the points in $\Omega^\tau$, where the four conjugate cyclic subgroups of order 3 in $C_N(\tau)$ each stabilising one of these four points in $\Omega^\tau$. Elements conjugate to $\tau\T$ or $\tau\T^{-1}$ behave similarly to $\T$ or $\T^{-1}$.
 
%Hall subgroup normalisers
\subsubsection{Normalisers of Hall subgroups $A_2, A_3$} The cyclic Hall subgroups, $A_2, A_3$, are normalised by cyclic subgroups of order 6. Let $A=\langle a \rangle$ be conjugate to a Hall subgroup $A_2$ or $A_3$ and let $N$ be its normaliser in $G$, the geometric picture is analogous in either case. Nontrivial elements of $A$ do not fix any points in $\Omega$ and $A$ is centralised only by the cyclic subgroup it generates, then as permutations they are composed of $(q^3+1)/\vert A \vert$ cycles of length $\vert A \vert$. Let $\T \in \C_3^*$ normalise $A$, then there are $\vert A \vert$ conjugates of $\T$ in $N$ and the fixed points of elements conjugate to $\T$ belong to the orbit of a unique cycle in $a$. For each conjugate of $\T$ there is an involution $\J$ with which is commutes and so the fixed point of $\T$ belongs to the block of $\J$. The remaining elements in the block of $\J$ each belong to a unique orbit of $a$ since if an orbit of $a$ contained more than one element of $\Omega^\J$ then $\J$ would commute with $a$. The elements of order 6, conjugate to $\J\T$ behave similarly to the elements conjugate to $\T$.

\subsection{Containments between Hall subgroup normalisers}
%MW
We make the following definition.
\begin{definition} Let $q=3^n$ be an odd power of 3. For a divsor $m$ of $n$ we define \[a_1(m)=3^m+1, \; \; a_2(m)=3^m-3^\frac{m+1}{2}+1, \; \; a_3(m)=3^m+3^\frac{m+1}{2}+1\] and in the case $m=n$ we omit the $(m)$. \end{definition}
\begin{remark} From the defintion we can write $a_1(m)a_2(m)a_3(m)=3^{3m}+1$ and if $G$ is a small Ree group, then a Hall subgroup of $G$ conjugate to $A_i$ has order $a_i/4$ for $i=1$ or $a_i$ for $i=2,3$. \end{remark}
If $G=R(3^{3m})$ then a Hall subgroup, $H$, conjugate to $A_1$ will have order $a_1(m)a_2(m)a_3(m)/4$ and if an element $h \in H$ belongs to $G_m$, then it will be normalised in $G_m$ either by a cyclic subgroup of order 6, or by a subgroup isomorphic to $2 \times L_2(3)$. Since $H$ is cyclic we need only prove the following number theoretic Lemma in order to aid the accurate determination of the overgroups of such an intersection.

\begin{lemma} \label{numb} Let $l$ be a positive factor of $n>2$ an odd natural number. Then one and only one of $a_1(n)$, $a_2(n)$ or $a_3(n)$ will be divisble by $a_i(l)$ for each  $i=1,2,3$.\end{lemma}

\begin{proof} Let $l$ and $n$ be as in the hypothesis and write $a_i$ for $a_i(n)$. It is clear that gcd$(a_i,a_j)=1$ for $i \neq j$ and so the $a_i(l)$ will divide at most one of the $a_i$. Also, $a_1(l)$ divides $a_1$ and if 3 divides $n/l$ then $a_1$ is divisible by $a_1(l)a_2(l)a_3(l)$, so assume that $i=2$ or 3 and that $n/l \equiv \pm1$ mod 3. Consider the values of $a_2(l)$ and $a_3(l)$ modulo $m=a_2(l)a_3(l)=3^{2l}-3^l+1$. We have that $a_1(l)n=3^{3l}+1$ and so $3^{3l} \equiv -1$ mod $m$ which gives us the following chain of congruences
\[3^n \equiv (-1)3^{n-3l} \equiv (-1)^23^{n-6l} \equiv \dots \equiv (-1)^k3^{n-3kl} \textnormal{ mod } m\] where $0\leq n-3kl < 3l$, from which it follows that \[3^n \equiv
\begin{cases}
(-1)^{\frac{n-l}{3l}}3^l=3^l			&\text{ mod } m, \quad	\text{if } (n/l) \equiv 1 \text{ mod } 3\\
(-1)^{\frac{n-2l}{3l}}3^{2l}=-3^{2l}		&\text{ mod } m, \quad	\text{if } (n/l) \equiv -1 \text{ mod } 3.\\
\end{cases}\]
Similarly, we have \[ 3^{\frac{n+1}{2}} \equiv \dots \equiv (-1)^k3^{\frac{n+1}{2}-3kl} \textnormal{ mod } m\] where this time $0 \leq \frac{n+1}{2}-3kl < 3l$. Eventually we find \[3^{\frac{n+1}{2}} \equiv
\begin{cases}
(-1)^{\frac{n-l}{6l}}3^{\frac{l+1}{2}}	&\text{ mod } m, \quad	\text{if } (n/l) \equiv 1 \text{ mod } 3\\
(-1)^{\frac{n-5l}{6l}}3^{\frac{5l+1}{2}}	&\text{ mod } m, \quad	\text{if } (n/l) \equiv -1 \text{ mod } 3.\\
\end{cases}\]
It can then be easily verified that \[
\frac{n-l}{6l} \equiv \frac{n-5l}{6l} \equiv
\begin{cases}
0 \text{ mod } 2		&\text{ if } (n/l) \equiv 1 \text{ mod } 4\\
1 \text{ mod } 2		&\text{ if } (n/l) \equiv 3 \text{ mod } 4.
\end{cases}\]
Assembling these results, along with the observation that 
%\[\mp3^{\frac{5f+1}{2}}-3^{2f}+1=(3^f\pm3^{\frac{f+1}{2}}+1)(\mp3^{\frac{3f+1}{2}}+3^{f+1}-3^f\mp3^{\frac{f+3}{2}}\pm2.3^{\frac{f+1}{2}}+1),\]
\[\mp3^{\frac{5l+1}{2}}-3^{2l}+1=(3^l+1\pm3^{\frac{l+1}{2}})(3^{l+1}-3^l+1\mp(3^{\frac{3l+1}{2}}+3^{\frac{l+3}{2}}-2.3^{\frac{l+1}{2}})),\]
 we finally arrive at the following
\[3^n \pm 3^{\frac{n+1}{2}} +1 \equiv
\begin{cases}
3^l \pm 3^{\frac{l+1}{2}} + 1		&\text{ mod } m \quad \text{if } (n/l) \equiv \pm1 \text{ mod } 12\\
3^l \mp 3^{\frac{l+1}{2}} + 1		&\text{ mod } m \quad \text{if } (n/l) \equiv \pm5 \text{ mod } 12.\\
\end{cases}\]
This completes the proof.\end{proof}

%%%%%%%%%%%%%%%%%%%%%%%%%%%---Max Ints---%%%%%%%%%%%%%%%%%%%%%%%%%%%
\section{Intersections of maximal subgroups} \label{maxint}
Our aim is to show that if a subgroup of $G=R(3^n)$ is equal to the intersection of a number of maximal subgroups of $G$, then it belongs to one of the following classes: \begin{enumerate}
\item $\R(l)$, subfield subgroups $R(3^l)$,
\item $\para(l)$, parabolic subgroups $(3^l)^{1+1+1}\colon(3^l-1)$,
\item $\cent_\J(l)$, involution centralisers $2 \times L_2(3^l)$, where $l>1$,
\item $\norm_V(l)$, four-group normalisers $(2^2 \times D_{(3^l+1)/2})\colon3$, where $l>1$,
\item $\norm_2(l)$, $\norm_3(l)$, normalisers of Hall subgroups $A_2, A_3$ of order $a_2(l), a_3(l) > 1$,
\item $\cent_V(l)$, four-group centralisers $2^2 \times D_{(3^l+1)/2}$ where $l>1$,
\item $\dee_H(l)$, a dihedral subgroup of order $2a_2(l)$ or $2a_3(l) \geq 14$, when $3l$ divides $n$,
\item $\cent_\J^\omega(l)$, point stabilisers of involution centralisers $2 \times (3^l\colon\frac{3^l-1}{2})$ where $l>1$,
\item $\eff(l)$, Sylow 3-subgroups of involution centralisers, elementary abelian of order $3^l$, $l>1$,
\item $\cent_0(l)$, centralisers of Hall subgroups $A_1$, cyclic of order $3^l-1$ where $l>1$,
\item $\cent_\J(1)$, involution centralisers in $R(3)$,
\item $\ee$, Sylow 2-subgroups of $G$, elementary abelian of order 8,
\item $\vier$, four-groups,
\item $\cent_6^*$, $\cent_3^*$, $\cent_2$, cyclic subgroups of orders 6, 3 and 2, generated by an element $\J\T$, $\T$ or $\J$,
\item $\eye$, the identity subgroup
\end{enumerate}
where in each class $l$ runs through all positive factors of $n$ unless otherwise stated.
In anticipation, we term the union of these classes $\mi$.
\begin{remark} There is a well known exceptional isomorphism between $L_2(3)$ and $A_4$, but in context it is more appropriate to think of it as $L_2(3)$. Where we omit the $(l)$ for classes (1)--(5) we mean the elements of that class that are maximal subgroups of $R(q)$, otherwise those elements for which $l=n$, if applicable. The exclusions made are to avoid the following repetitions \[\norm_V(1) = \cent_\J(1), \; \; \cent_V(1) = \ee, \; \; \cent_\J^\omega (1) = \cent_6^*, \; \; \eff(1) = \cent_3^*, \; \; \dee_H(1) \supset \cent_0(1)  = \cent_2.\] In particular, the list is ordered such that no element of a class appears in more than one class, and that no element of a class is a subgroup of any element of a successive class in the stated ordering. \end{remark}

We devote the rest of this section to proving the following theorem, since we can immediately determine by Theorem \ref{hall} that $\mu_G(H)=0$ for any subgroup $H$ which does not occur as the intersection of maximal subgroups.

\begin{theorem} \label{int} Let $G$ be a simple small Ree group and let $H$ be a subgroup that is the intersection of a number of maximal subgroups of $G$. Then $H \in \mi$. \end{theorem}

We begin by determining the intersections of pairs of maximal subgroups in $G$.

%R...
\subsection{Intersections with subfield subgroups} \label{rint}
Let $G_m$ be a maximal subfield subgroup of $G$ and $\Omega(m)$ the $3^{3m}+1$ points in $\Omega$ fixed by the Sylow 3-subgroups of $G_m$.

%R \cap P
\begin{lemma} The intersection of $G_m$ with a parabolic subgroup belongs to $\para(l) \cup \cent_2 \cup \eye$. \end{lemma}
\begin{proof} Let $P_\omega$ be the stabiliser of $\omega \in \Omega$. If $\omega \in \Omega(m)$ then the intersection of $G_m$ with $P_\omega$ will be the stabiliser of $\omega$ in $G_m$, belonging to $\para(l)$. Otherwise, since only involutions in $G_m$ can fix an $\omega \notin \Omega(m)$, we get a subgroup of an element of $\cent_2$. \end{proof}

%R \cap C
\begin{lemma} \label{sc} The intersection of $G_m$ with an involution centraliser belongs to $\cent_\J(l) \cup \eff(l) \cup \dee_H(l) \cup \vier \cup \cent_3^* \cup \cent_2 \cup \eye$. \end{lemma}
\begin{proof} Let $\J \in G$ be an involution, $C=C_G(\J)$ its centraliser in $G$ and $H=G_m \cap C$, stabilising $\Omega(m) \cup \Omega^\J$. If $\Omega^\J$ intersects $\Omega(m)$ in at least 2 points then $\J \in G_m$ and $H$ is the centraliser in $G_m$ of $\J$ belonging to $\cent_\J(l)$. If $\Omega^\J \cap \Omega(m) = \{\omega\}$ then $\J \notin G_m$ and the only nontrivial elements of $C$ which can fix $\omega$ are the elements of order 3, so that $H \in \eff(l) \cup \cent_3^* \cup \eye$.

Suppose now that $\Omega^\J \cap \Omega(m) = \varnothing$. Then $\J \notin G_m$ and $H$ is isomorphic to a subgroup of $L_2(q)$ not containing elements of order 3, $k$ or $2k$ where $k>3$ divides $(q-1)/2$, hence $H$ is a subgroup of $D_{q+1}$. If $H$ is nontrivial and does not contain elements of order $k>2$ dividing $(q+1)/4$ then $H \sg V$ for some $V \in \vier$ and belongs to our list. If there exists $s \in H$ of order $k$ then $k$ divides $a_1(m)/4$ or, if 3 divides $m$, $a_2(m/3)a_3(m/3)$. In the former case, the involutions centralising $s$ belong to $G_m$ so that $s \notin H$. In the latter case, involutions which centralise $s$ do not belong to $G_m$, but involutions which normalise $s$ in $G$ will belong to $G_m$, hence $H$ is isomorphic to $D_{2a_2(m/3)}$ or $D_{2a_3(m/3)}$.\end{proof}

%R \cap Hall
\begin{lemma} The intersection of $G_m$ with an element of $\norm_V \cup \norm_2 \cup \norm_3$ belongs to $\norm_V(l) \cup \norm_2(l) \cup \norm_3(l) \cup \cent_6^* \cup \cent_3^* \cup \cent_2 \cup \eye$. \end{lemma}
\begin{proof} Let $a$ generate any Hall subgroup conjugate to $A_i$, where $i=1,2,3$. If $a \in G_m$, then the intersection is equal to the normaliser in $G_m$ of $a$ which belongs to $\norm_V(l) \cup \norm_2(l) \cup \norm_3(l)$. If $a \notin G_m$ then, since the centraliser in $G$ of $a$ is uniquely contained in its normaliser in $G$, we have that the intersection is a subgroup of $N_G(a)/C_G(a) \cong C_6$, and so is a subgroup of an element of $\cent_6^*$.
\end{proof}

%R \cap R
\begin{lemma} The intersection of two distinct maximal subfield subgroups belongs to our list. \end{lemma}
\begin{proof} Let $G_m \neq G_l$ be maximal subfield subgroups of $G$. Their intersection will be a subgroup of a maximal subgroup of $G$. If their intersection belongs to $\R(l)$, then it belongs to $\mi$, otherwise their intersection will be a subgroup of $H \in \para \cup \cent_\J \cup \norm_V \cup \norm_2 \cup \norm_3$. In particular, it will be the intersection of $H \cap G_m$ with $G_l$ which we have dealt with in the preceding lemmas in this Section. \end{proof}

%P...
\subsection{Intersections with parabolic subgroups} \label{pint}

%P \cap P
\begin{lemma} The intersection of two distinct parabolic subgroups belongs to $\cent_0$. \end{lemma}
\begin{proof} The Sylow 3-subgroups of $G$ have trivial intersection and anything lying in two distinct parabolic subgroups must pointwise fix two points. \end{proof}

%P \cap C_G(i)
\begin{lemma} The intersection of a parabolic subgroup with an involution centraliser belongs to $\cent_\J^\omega \cup \cent_2$. \end{lemma}
\begin{proof} Let $P_\omega$ be the stabiliser of $\omega \in \Omega$ and for an involution $\J \in G$ let $C=C_G(\J)$ its centraliser in $G$. Since all involutions in $P$ are conjugate we need only distinguish the cases when $\J$ is an element of $P_\omega$ or not. If $\J \in P_\omega$ then $\omega \in \Omega^\J$ and the intersection is the stabiliser of $\omega$ in $C$, isomorphic to $2 \times (q\colon\frac{q-1}{2})$. If $\J \notin P_\omega$ then $\omega \notin \Omega^\J$ and the only elements in $C$ which will share a fixed point with $P_\omega$ are the involutions. Since the blocks of involutions in $C$ are disjoint and since every element of $\Omega$ belongs to the block of some involution in $C$, the intersection will belong to $\cent_2$. \end{proof}

%P \cap N
\begin{lemma} The intersection of a parabolic subgroup with an element of $\norm_V \cup \norm_2 \cup \norm_3$ belongs to $\cent_6^* \cup \cent_3^* \cup \cent_2 \cup \eye$. \end{lemma}
\begin{proof} This is clear from the comparison of orders. \end{proof}

%C...
\subsection{Intersections with involution centralisers} \label{cint}

%C_G(i) \cap C_G(j)
\begin{lemma} The intersection of two distinct involution centralisers belongs to $\cent_\vier \cup \eff \cup \vier \cup \cent_2 \cup \eye$. \end{lemma}
\begin{proof} Let $\J \neq \J'$ be involutions in $G$ and $C=C_G(\J)$. If $\J' \in C$ then $H = C \cap C_G(\J') = C_G(\langle \J,\J' \rangle) \in \cent_V$, so suppose that $\J' \notin C$. If $\Omega^\J \cap \Omega^{\J'} = \{\omega\}$ and $H$ is nontrivial, then it will equal the 3-group in $\eff$ belonging to both involution centralisers. If $\Omega^\J \cap \Omega^{\J'} = \varnothing$ then their intersection will not contain elements of order 3 or dividing $(q-1)/2$ since they fix points in $\Omega^\J \cap \Omega^{\J'}$. If $s \in C$ is an element of order $k>2$ dividing $(q+1)/4$, then its centraliser in $G$ is equal to its centraliser in $C_G(\J)$ and is isomorphic to $\langle s \rangle \times 2^2$, then, since $\J' \notin C$, elements of order $k$ cannot belong to $C \cap C_G(\J')$ and the intersection is then a subgroup of a four-group.\end{proof}

%C_G(i) \cap N_G(V)
\begin{lemma} The intersection of an involution centraliser with a four-group normaliser belongs to $\cent_\vier \cup \cent_\J(1) \cup \ee \cup \cent_3^* \cup \eye$. \end{lemma}
\begin{proof} Let $V$ be a four-group in $G$ and $N=N_G(V)$ its normaliser in $G$. If $\J \in V$ then $N \cap C_G(\J)$ is the centraliser in $N$ of $\J$ which belongs to $\cent_V$. Elements in $G$ of order dividing $(q+1)/4$ belong to a unique four-group normaliser and will only belong to the intersection of $N$ with a centraliser of an involution in $N$, so all other intersections are isomorphic to a subgroup of $2 \times L_2(3)$. If $\J \in N \setminus V$ then the intersection belongs to an element of $\cent_\J(1)$ or $\ee$ depending on whether $\J$ is contained in a dihedral subgroup of $N$ that is normal in $N$ or not, as per the discussion in Section \ref{struct}. Finally, if $\J \notin N$ then the intersection is isomorphic to a subgroup of $2 \times L_2(3)$ which does not contain its central involution or its normal four-group, hence is an element of $\cent_3^* \cup \eye$. \end{proof}

\begin{lemma} The intersection of an involution centraliser with an element of $\norm_2 \cup \norm_3$ belongs to $\cent_6^* \cup \cent_3^* \cup \cent_2 \cup \eye$. \end{lemma}
\begin{proof} This is clear from the comparison of orders. \end{proof}

%Hall...
\subsection{Intersections with normalisers of four-groups or Hall subgroups} \label{nint}

\begin{lemma} The intersection of two distinct four-group normalisers belongs to $\ee \cup \cent_6^* \cup \cent_3^* \cup \cent_2 \cup \eye$. \end{lemma}
\begin{proof} Recall that the normaliser of a four-group is the normaliser of its normal subgroup of order $a_1=(q+1)/4$ conjugate to a cyclic Hall subgroup $A_1$. Let $s_1, s_2 \in G$, be elements of order $a_1$ belonging to distinct Hall subgroups. Since the centralisers in $G$ of $s_1$, $s_2$, belong to unique four-group normalisers, the intersection of $N_G(s_1) \cap N_G(s_2)$ is isomorphic to a subgroup of $N_G(s_1)/C_G(s_1) \in \cent_6^*$.
\end{proof}

\begin{lemma} The intersection of an element of $\norm_V \cup \norm_2 \cup \norm_3$ with a distinct element of $\norm_2 \cup \norm_3$ belongs to $\cent_6^* \cup \cent_3^* \cup \cent_2 \cup \eye$. \end{lemma}
\begin{proof} This is clear from comparison of the orders of the various groups and the observation that cyclic Hall subgroups have trivial intersection and belong to a unique Hall subgroup normaliser in $G$. \end{proof}

%Conj/Norm
\subsection{Conjugacy and normalisers of subgroups in $R(q)$} \label{cnr}

An important step in determing the \M function is determing the size and number of conjugacy classes of subgroups in each class which we now prove in the following Lemmas.

\begin{lemma} Elements of $\R(l)$ are self-normalising and isomorphic elements are conjugate in $G$. \end{lemma}
\begin{proof} Subfield subgroups are contained only in larger subfield subgroups which are simple, hence the former statement holds, the latter is proved in \cite{ln}. \end{proof}

\begin{lemma} Elements of $\para(l)$ are self-normalising and isomorphic elements are conjugate in $G$. \end{lemma}
\begin{proof} Elements of $\para(l)$ are contained in subfield subgroups or their parabolic subgroups and they are normal in neither. If two elements of $\para(l)$ have the same order then they are maximal subgroups of isomorphic, and therefore conjugate, subfield subgroups and so are themselves conjugate in $G$. \end{proof}

\begin{lemma} Elements of $\cent_\J(l)$ are self-normalising and isomorphic elements are conjugate in $G$. \end{lemma}
\begin{proof} The former follows from the fact that the normaliser must centralise the central involution and hence belongs to an involution centraliser. These have two nontrivial normal subgroups, neither of which belong to $\cent_\J(l)$. The latter follows from the conjugacy in $G$ of the subfield subgroups and of their involution centralisers. \end{proof}

\begin{lemma} Elements of $\norm_V(l) \cup \norm_2(l) \cup \norm_3(l)$ are all self-normalising and isomorphic elements are conjugate in $G$. \end{lemma}
\begin{proof} Let $N \in \norm_V(l) \cup \norm_2(l) \cup \norm_3(l)$. The subgroup $A \sg N$ contained in a Hall subgroup conjugate to $A_1$, $A_2$ or $A_3$ is characteristic, and so anything which normalises $N$ must normalise $A$, but since the cyclic subgroups of order 6 in $N$ are self-normalising, then so is $N$. The latter follows from the conjugacy in $G$ of isomorphic Hall subgroups and of isomorphic subfield subgroups.\end{proof}

\begin{lemma} Elements of $\cent_V(l)$ are normalised in $G$ by elements of $\norm_V(l)$ and isomorphic elements are conjugate in $G$. \end{lemma}
\begin{proof} The normaliser of an element $H \in \cent_V(l)$ must normalise its central four-group and so is contained in a four-group normaliser. Since $H$ is not normal in a four-group centraliser but will be normalised by an element of order 3, the normaliser of $H$ in $G$ belongs to $\norm_V(l)$. Any two isomorphic elements of $\cent_V(l)$ are centralisers of four-groups in conjugate subfield subgroups, from which it follows that they are conjugate in $G$.\end{proof}

\begin{lemma} Elements of $\dee_H(l)$ are normalised in $G$ by a subgroup isomorphic to either $D_{2a_i(l)}:3$ or $(2^2 \times D_{2a_i(l)})\colon3$, where $i=2,3$, and isomorphic elements are conjugate in $G$.\end{lemma}
\begin{proof} If $D \in D_H(l)$ is isomorphic to $D_{2a_2(l)}$ or $D_{2a_3(l)}$ then $D$ belongs to the normaliser of a Hall subgroup conjugate to $A_i$ for $i=1,2$ or 3. Since the normal Hall subgroup in $D$ is normalised in $G$ by $\T$, an element of order 3, then it is also normalised by $\J$, an element of order 2 that commutes with $\T$. By Lemma \ref{numb}, $D$ may additionally be centralised in $G_m \sg G$, a subfield subgroup, by a four-group according to the parity of $m/l \equiv$ mod 3.\end{proof}

\begin{lemma} Elements of $\cent_\J^\omega(l)$ are self-normalising and isomorphic elements are conjugate in $G$. \end{lemma}
\begin{proof} The normaliser in $G$ of $H \in \cent_\J^\omega(l)$ must fix the central involution and so is contained in an involution centraliser. Since $H$ is not a normal subgroup of such a subgroup we have that $H$ is self-normalising. Conjugacy of isomorphic elements follows from the conjugacy of subfield subgroups and their involution centralisers. \end{proof}

\begin{lemma} Elements of $\eff(l)$ are normalised in $G$ by subgroups isomorphic to $3^{2n}\colon(3^d-1)$ and isomorphic elements are conjugate in $G$. \end{lemma}
\begin{proof} If $H \in \eff(l)$ has order $3^h$, then $H$ is isomorphic to the additive subgroup of $L_2(3^h) \sg G$. The centraliser of $H$ in $G$ is the normal abelian subgroup of order $q^2$ of the Sylow 3-subgroup of $G$ to which $H$ belongs \cite[Chapter III]{ward}. It can be verified by direct calculation with the $7 \times 7$ matrices given in \cite{ln} that the elements normalising, but not centralising, $H$ in $G$ are those of a cyclic subgroup of order $3^h-1$, isomorphic to the multiplicative subgroup of $L_2(3^h)$. The conjugacy is clear, since they are normalisers of conjugate Sylow 3-subgroups. \end{proof}

\begin{lemma} Elements of $\cent_0(l)$ are normalised in $G$ by a subgroup isomorphic to $2 \times D_{q-1} \cong D_{2(q-1)}$ and isomorphic elements are conjugate in $G$. \end{lemma}
\begin{proof} If $H \in \cent_0(l)$ then it is centralised in $G$ by a unique Hall subgroup conjugate to $A_0$ and its normaliser in $G$ belongs to the centraliser of the unique involution in $H$. Hence the normaliser in $G$ of $H$ is $N_G(H) \cong D_{2(q-1)}$. The conjugacy of elements of $\cent_0(l)$ follows from the fact that they are normal and unique in the Hall subgroups to which they belong, and that the Hall subgroups $A_0$ are conjugate in $G$. \end{proof}

\begin{lemma} Elements of $\cent_\J(1)$ are conjugate in $G$ and self-normalising. \end{lemma}
\begin{proof} The normaliser of an element of $\cent_\J(1)$ must centralise its central involution and so is contained in an involution centraliser, in which it is not normal. The conjugacy follows from the conjugacy of involutions.\end{proof}

Most of the proofs of the following have already been established, we state them for completeness.

\begin{lemma} Elements of $\ee$ are conjugate in $G$ and normalised by a subgroup isomorphic to $2^3\colon7\colon3$. Elements of $\vier$ are conjugate in $G$ and normalised by a subgroup isomorphic to $(2^2 \times D_\frac{q+1}{2})\colon3$. \end{lemma}
\begin{proof} Ward \cite{ward} proves that 2-subgroups of the same order are conjugate in $G$ and the normaliser of a Sylow 2-subgroup has shape $2^3:7:3$. The rest is clear. \end{proof}

\begin{lemma} The following are conjugacy classes of subgroups of $G$: \begin{enumerate}
\item the class $\cent_6^*$, elements of which are normalised in $G$ by a subgroup of order $2q$,
\item the class $\cent_3^*$, elements of which are normalised in $G$ by a subgroup of order $2q^2$, and
\item the class $\cent_2$, elements of which are normalised in $G$ by a subgroup of order $q(q^2-1)$.
\end{enumerate}\end{lemma}
\begin{proof} Elements of subgroups of orders 2, 3, or 6 are not conjugate in $G$, hence their normalisers in $G$ are equal to their centralisers in $G$, which have been established. \end{proof}

%Keystone
\subsection{The proof of Theorem \ref{int}}

We have so far determined what the intersection of two maximal subgroups in $R(q)$ can be. The final step in proving Theorem \ref{int} is to show that the intersection of more than two maximal subgroups does not yield a subgroup not already on our list.

\begin{proof}[Proof of Theorem \ref{int}]
By the Lemmas in Section \ref{rint} it is clear that the intersection of a pair of maximal subgroups with a number of subfield subgroups will not add anything new to the classes obtained there. By the Lemmas in Sections \ref{rint}--\ref{pint} and the fact that the intersection of three distinct parabolic subgroups is the identity subgroup we have that the intersection of a pair of maximal subgroups with a parabolic subgroup contributes the elements of $\cent_\J^\omega(l)$ where $f<e$ to the classes obtained in Section \ref{pint}. By the Lemmas in Sections \ref{rint}--\ref{cint} and the fact that the intersection of more than three involution centralisers is a Sylow 2-subgroup of $G$ the only new classes we obtain are the elements of $\eff(l)$ where $f=e$. Finally, it is clear that any further intersection with the Hall subgroup normalisers does not contribute any new subgroups.\end{proof}

%%%%%%%%%%%%%%%%%%%%%%%%%%%---M\"oefficients---%%%%%%%%%%%%%%%%%%%%%%%%%%%
\section{Determination of the \M function of $R(q)$} \label{det}
We follow closely the style used by Downs \cite{psl2} in order to calculate $\mu_G(H)$ for $H \sg G$. That is, for a fixed subgroup conjugate to $H$ we consider the overgroups $H \sg K \sg G$ for which $\mu_G(K) \neq 0$ to determine their contribution to $\mu_G(H)$. We use a counting argument to determine the number of subgroups conjugate to $K$ which contain $H$. For the most part this is relatively straightforward but in some of the smaller cases, such as the Sylow 2-subgroups or the four-groups, these split into distinct conjugacy classes within a given overgroup and so we use the following formula when we count. Let $N(K,H)$ be the number of subgroups conjugate to $H$ in $G$, contained in $K$. Then, $H$ is contained in
\[\nu_K(H)=\frac{[G:N_G(K)]N(K,H)}{[G:N_G(H)]}\]
subgroups conjugate to $K$ in $G$, where $[G:K]$ is the index of a subgroup $K\sg G$. In what follows, $\mu(n)$, for a positive integer, $n$, will be the classical number theoretic \M function defined by\[\mu(n)= \begin{cases}
1		&\text{if $n=1$}\\
(-1)^d	&\text{if $n$ is the product of $d$ distinct primes}\\
0		&\text{if $n$ has a square factor}.\\
\end{cases}\]

%R(3^f)
\subsection{$H\cong R(3^h) \in \R(l)$}

A simple counting argument shows that for a subfield subgroup $R(3^h)$, the subfield subgroups which contain it are in one-to-one correspondence with the divisors in the lattice of divisors, $k$, where $h \vert k \vert e$. It is then clear that $\mu_G(H)=\mu(n/k)$.

%Parabolic
\subsection{$H \cong (3^h)^{1+1+1}\colon(3^h-1) \in \para(l)$}
\begin{center}\begin{tabular}{c c c c c}\hline
Isomorphism type				&for $k \vert n$			&									&					&\\
of overgroup $K$				&and s.t.				&$N_K(H)$							&$\nu_K(H)$		&$\mu_G(K)$\\ \hline
$R(3^k)$						&--						&$(3^h)^{1+1+1}\colon(3^h-1)$		&1					&$\mu(n/k)$\\
$(3^k)^{1+1+1}\colon(3^k-1)$	&$k>h$					&$(3^h)^{1+1+1}\colon(3^h-1)$		&1					&$-\mu(n/k)$\\ \hline
\end{tabular}\end{center}
For $k>h$ we have the contribution of $R(3^k)$ cancelling with the contribution of $(3^k)^{1+1+1}\colon(3^k-1)$ and so $\mu_G(H)=-\mu(n/h)$.

%Involution centraliser
\subsection{$H \cong 2 \times L_2(3^h) \in \cent_\J(l)$}
\begin{center}\begin{tabular}{c c c c c}\hline
Isomorphism type			&for $k \vert n$			&							&					&\\
of overgroup $K$			&and s.t.				&$N_K(H)$					&$\nu_K(H)$		&$\mu_G(K)$\\ \hline
$R(3^k)$					&--						&$2 \times L_2(3^h)$		&1					&$\mu(n/k)$\\
$2 \times L_2(3^k)$			&$k>h$					&$2 \times L_2(3^h)$		&1					&$-\mu(n/k)$\\ \hline
\end{tabular}\end{center}
As in the case of the parabolic subgroups, the contribution from the subfield subgroups cancel with each $k>h$ and we are left with $\mu_G(H)=-\mu(n/h)$.

%N_G(V)
\subsection{$H \cong (2^2 \times D_{(3^h+1)/2})\colon3 \in \norm_V(l)$}
\begin{center}\begin{tabular}{c c c c c}\hline
Isomorphism type						&for $k \vert n$			&										&				&\\
of overgroup $K$						&and s.t.				&$N_K(H)$								&$\nu_K(H)$	&$\mu_G(K)$\\ \hline
$R(3^k)$								&--						&$(2^2 \times D_{(3^h+1)/2})\colon3$	&1				&$\mu(n/k)$\\
$(2^2 \times D_{(3^k+1)/2})\colon3$	&$k>h$					&$(2^2 \times D_{(3^h+1)/2})\colon3$	&1				&$-\mu(n/k)$\\ \hline
\end{tabular}\end{center}
As in the case of the parabolic subgroups and involution centralisers, we find that for $H \in \norm_V(l)$, $\mu_G(H) = -\mu(n/k)$.

%MW
\subsection{$H \in \norm_2(l) \cup \norm_3(l)$}
Let $H$ be conjugate to $3^h-3^\frac{h+1}{2}+1\colon 6$ or $3^h+3^\frac{h+1}{2}+1\colon 6$ in $\norm_2(l) \cup \norm_3(l)$.
\begin{center}\begin{tabular}{c c c c c}\hline
Isomorphism type						&for $k \vert n$										&				&				&\\
of overgroup $K$						&and s.t.											&$N_K(H)$		&$\nu_K(H)$	&$\mu_G(K)$\\ \hline
$R(3^k)$								&--													&$H$			&1				&$\mu(n/k)$\\
$(2^2 \times D_{(3^k+1)/2})\colon3$	&$\frac{k}{h} \equiv 0 \textnormal{ mod } 3$			&$H$			&1				&$-\mu(n/k)$\\
$3^k+\sqrt{3^{k+1}}+1\colon6$			&$\frac{k}{h} \equiv 1,5 \textnormal{ mod } 12$		&$H$			&1				&$-\mu(n/k)$\\
$3^k-\sqrt{3^{k+1}}+1\colon6$			&$\frac{k}{h} \equiv -1,-5 \textnormal{ mod } 12$	&$H$			&1				&$-\mu(n/k)$\\ \hline
\end{tabular}\end{center}
For each $k$ such that $h \vert k \vert n$ we have by Lemma \ref{numb} that $H$ is contained in one, and only one, of an overgroup belonging to $\norm_V(l)$, $\norm_2(l)$ or $\norm_3(l)$. Then for each $k>h$, the contribution from $R(3^h)$ cancels with a contribution from one of these overgroups. We then see that $\mu_G(H)=-\mu(n/h)$.

%C_G(V)
\subsection{$H \cong 2^2 \times D_{(3^h+1)/2} \in \cent_V(l)$}
\begin{center}\begin{tabular}{c c c c c}\hline
Isomorphism type						&for $k \vert n$		&											&					&\\
of overgroup $K$						&and s.t.			&$N_K(H)$									&$\nu_K(H)$		&$\mu_G(K)$\\ \hline
$R(3^k)$								&--					&$(2^2 \times D_{(3^h+1)/2})\colon3$		&1					&$\mu(n/k)$\\
$(2^2 \times D_{(3^k+1)/2})\colon3$	&--					&$(2^2 \times D_{(3^h+1)/2})\colon3$		&1					&$-\mu(n/k)$\\
$2 \times L_2(3^k)$						&--					&$2^2 \times D_{(3^h+1)/2}$				&3					&$-\mu(n/k)$\\
$2^2 \times D_{(3^k+1)/2}$				&$k>h$				&$2^2 \times D_{(3^h+1)/2}$				&1					&$3\mu(n/k)$\\ \hline
\end{tabular}\end{center}
For each $h$ the contribution from the subfield subgroups cancels with that of the four-group normalisers leaving the contribution coming from the involution centralisers and so $\mu_G(H)=3\mu(n/h)$.

%Those Dihedrals...
\subsection{$H \in \dee_H(l)$} Let $H \in \dee_H(l)$ be isomorphic to $D_{2a_2(h)}$ or $D_{2a_3(h)}$ for a divisor $h$ such that $3h \vert n$.
\begin{center}\begin{tabular}{c c c c c c c}\hline
Isomorphism type						&for $k \vert n$									&							&				&\\
of overgroup $K$						&and s.t.										&$N_K(H)$					&$\nu_K(H)$	&$\mu_G(K)$\\ \hline
$R(3^k)$								&$\frac{k}{h} \equiv 0 \textnormal{ mod } 3$		&$(2^2 \times H)\colon3$	&1				&$\mu(n/k)$\\
$(2^2 \times D_{(3^k+1)/2})\colon3$		&$\frac{k}{h} \equiv 0 \textnormal{ mod } 3$		&$(2^2 \times H)\colon3$	&1				&$-\mu(n/k)$\\
$2 \times L_2(3^k)$						&$\frac{k}{h} \equiv 0 \textnormal{ mod } 3$		&$2^2 \times H$			&3				&$-\mu(n/k)$\\
$2^2 \times D_{(3^k+1)/2}$				&$\frac{k}{h} \equiv 0 \textnormal{ mod } 3$		&$2^2 \times H$			&1				&$3\mu(n/k)$\\ \hline
$R(3^k)$								&$\frac{k}{h} \equiv \pm1 \textnormal{ mod } 3$	&$H\colon3$				&4				&$\mu(n/k)$\\
$3^k\pm\sqrt{3^{k+1}}+1\colon6$			&$\frac{k}{h} \equiv \pm1 \textnormal{ mod } 3$	&$H\colon3$				&4				&$-\mu(n/k)$\\ \hline
\end{tabular}\end{center}
From the table it is clear that for each divisor $k$, such that $h \vert k \vert n$, the contributions to $\mu_G(H)$ cancel with one another and so $\mu_G(H)=0$.

%Parabolic + Involution Centraliser
\subsection{$H \cong 2 \times (3^h\colon\frac{3^h-1}{2}) \in \cent_\J^\omega(l)$}
\begin{center}\begin{tabular}{c c c c c}\hline
Isomorphism type						&for $k \vert n$			&												&					&\\
of overgroup $K$						&and s.t.				&$N_K(H)$										&$\nu_K(H)$		&$\mu_G(K)$\\ \hline
$R(3^k)$								&--						&$2 \times (3^{h} \colon \frac{3^h-1}{2})$		&1					&$\mu(n/k)$\\
$(3^k)^{1+1+1}\colon(3^k-1)$			&--						&$2 \times (3^{h} \colon \frac{3^h-1}{2})$		&1					&$-\mu(n/k)$\\
$2 \times L_2(3^k)$						&--						&$2 \times (3^{h} \colon \frac{3^h-1}{2})$		&1					&$-\mu(n/k)$\\
$2 \times (3^{k}\colon\frac{3^k-1}{2})$	&$k>h$					&$2 \times (3^{h} \colon \frac{3^h-1}{2})$		&1					&$\mu(n/k)$\\ \hline 
\end{tabular}\end{center}
For $k>h$ the contribution from the subfield subgroups and parabolc subgroups cancel, leaving the contribution from the involution centralisers, and so $\mu_G(H)=\mu(n/h)$.

%Sylow 3 of the involution centraliser
\subsection{$H \cong 3^h \in \eff(l)$}
\begin{center}\begin{tabular}{c c c c c}\hline
Isomorphism type						&for $k \vert n$		&											&						&\\
of overgroup $K$						&and s.t.			&$N_K(H)$									&$\nu_K(H)$			&$\mu_G(K)$\\ \hline
$R(3^k)$								&--					&$3^{2k}\colon3^h-1$						&$3^{2(n-k)}$			&$\mu(n/k)$\\
$(3^k)^{1+1+1}\colon(3^k-1)$			&--					&$3^{2k}\colon3^h-1$						&$3^{2(n-k)}$			&$-\mu(n/k)$\\
$2 \times L_2(3^k)$						&--					&$2 \times (3^k\colon\frac{3^h-1}{2})$		&$3^{2n-k}$			&$-\mu(n/k)$\\
$2 \times (3^k\colon\frac{3^k-1}{2}$)	&--					&$2 \times (3^k\colon\frac{3^h-1}{2})$		&$3^{2n-k}$			&$\mu(n/k)$\\
$3^{k}$									&$k>h$				&$3^k$										&$(3^k-1)/(3^h-1)$		&0\\ \hline 
\end{tabular}\end{center}
We see that the contributions from the subfield subgroups cancel with those of the parabolic subgroups, and similarly the contributions from the involution centralisers cancel with the contributions from their intersection with the parabolic subgroups, giving $\mu_G(H)=0$.

%Stabiliser of 2 points
\subsection{$H \cong 3^h-1 \in \cent_0(l)$}
\begin{center}\begin{tabular}{c c c c c}\hline
Isomorphism type						&for $k \vert n$		&						&					&\\
of overgroup $K$						&and s.t.			&$N_K(H)$				&$\nu_K(H)$		&$\mu_G(K)$\\ \hline
$R(3^k)$								&--					&$D_{2(3^k-1)}$		&1					&$\mu(n/k)$\\
$2 \times L_2(3^k)$						&--					&$D_{2(3^k-1)}$		&1					&$-\mu(n/k)$\\
$(3^k)^{1+1+1}\colon3^k-1$			&--					&$3^k-1$				&2					&$-\mu(n/k)$\\
$2 \times (3^k\colon\frac{3^k-1}{2})$	&--					&$3^k-1$				&2					&$\mu(n/k)$\\
$3^k-1$								&$k>h$				&$3^k-1$				&1					&0\\ \hline
\end{tabular}\end{center}
For each $h$ we see that the contributions from the first four classes of overgroups cancel, leaving $\mu_G(H)=0$.

%2xL_2(3)
\subsection{$H \cong 2 \times L_2(3) \in \cent_\J(1)$}
\begin{center}\begin{tabular}{c c c c c}\hline
Isomorphism type						&for $k \vert n$		&							&					&\\
of overgroup $K$						&and s.t.			&$N_K(H)$					&$\nu_K(H)$		&$\mu_G(K)$\\ \hline
$R(3^k)$								&--					&$2 \times L_2(3)$			&1					&$\mu(n/k)$\\
$2 \times L_2(3^k)$						&$h>1$				&$2 \times L_2(3)$			&1					&$-\mu(n/k)$\\
$(2^2 \times D_{(3^k+1)/2})\colon3$	&$h>1$				&$2 \times L_2(3)$			&1					&$-\mu(n/k)$\\ \hline
\end{tabular}\end{center}
The summation of the $R(3^k)$ is equal to 0 and so the remainder of the remaining two lines means that $\mu_G(2 \times L_2(3))=-2\mu(n)$.

%2^3
\subsection{$H \cong 2^3 \in \ee$}
\begin{center}\begin{tabular}{c c c c c}\hline
Isomorphism type						&for $k \vert n$		&							&					&\\
of overgroup $K$						&and s.t.			&$N_K(H)$					&$\nu_K(H)$		&$\mu_G(K)$\\ \hline
$R(3^k)$								&--					&$2^3\colon7\colon3$		&1					&$\mu(n/k)$\\
$2 \times L_2(3^k)$						&$k>1$				&$2 \times L_2(3)$			&7					&$-\mu(n/k)$\\
$(2^2 \times D_{(3^k+1)/2})\colon3$	&$k>1$				&$2 \times L_2(3)$			&7					&$-\mu(n/k)$\\
$2 \times L_2(3)$						&--					&$2 \times L_2(3)$			&7					&$-2\mu(n)$\\
$2^2 \times D_{(3^k+1)/2}$				&$k>1$				&$2^3$						&7					&$3\mu(n/k)$\\ \hline
\end{tabular}\end{center}
The summation over the $R(3^k)$ equates to 0, as does the total summation over the succeeding three classes. From summation over the final line we then have that $\mu_G(2^3)=21\mu(n)$.

%V
\subsection{$H \cong 2^2 \in \vier$}
\begin{center}\begin{tabular}{c c c c c}\hline
Isomorphism type						&for $k \vert n$		&											&							&\\
of overgroup $K$						&and s.t.			&$N_K(H)$									&$\nu_K(H)$				&$\mu_G(K)$\\ \hline
$R(3^k)$								&--					&$(2^2 \times D_{(3^k+1)/2})\colon3$		&$(3^n+1)/(3^k+1)$		&$\mu(n/k)$\\
$(2^2 \times D_{(3^k+1)/2})\colon3$	&$k>1$				&$(2^2 \times D_{(3^k+1)/2})\colon3$		&$(3^n+1)/(3^k+1)$		&$-\mu(n/k)$\\
$2^2 \times D_{(3^k+1)/2}$				&$k>1$				&$2^2 \times D_{(3^k+1)/2}$				&$(3^n+1)/(3^k+1)$		&$3\mu(n/k)$\\
$2 \times L_2(3^k)$						&$k>1$				&$2^2 \times D_{(3^k+1)/2}$				&$3(3^n+1)/(3^k+1)$		&$-\mu(n/k)$\\
$(2^2 \times D_{(3^k+1)/2})\colon3$	&$k>1$				&(2) $2^3$									&$3(3^n+1)/2$				&$-\mu(n/k)$\\
$2^2 \times D_{(3^k+1)/2}$				&$k>1$				&(6) $2^3$									&$3(3^n+1)/2$				&$3\mu(n/k)$\\
$2 \times L_2(3^k)$						&$k>1$				&(1) $2^3$, (1) $2 \times L_2(3)$			&$3^n+1$					&$-\mu(n/k)$\\
$2 \times L_2(3)$						&--					&(2) $2^3$, (1) $2 \times L_2(3)$			&$7(3^n+1)/4$				&$-2\mu(n)$\\
$2^3$									&--					&(7) $2^3$									&$(3^n+1)/4$				&$21\mu(n)$\\
\hline
\end{tabular}\end{center}
Four-groups are conjugate in $G$ but not necessarily conjugate in subgroups of $G$. Where this is the case, in the $N_K(H)$ column the number in parentheses denotes the number of conjugacy classes of $V$ whose normaliser in $K$ is of the specified isomorphism type. In order to make verification of the arithmetic a little easier, we have separated contributions from overgroups isomorphic to $K$ according to whether the contribution depends on $k$ or not. In the cases where there is no dependence on $k$ the usual properties of the \M function leave us a few terms to tidy up and we eventually find that $\mu_G(2^2)=0$.

%C_6
\subsection{$H \cong \langle \J\T \rangle \in \cent_6^*$}
\begin{center}\begin{tabular}{c c c c c}\hline
Isomorphism type							&for $k \vert n$		&							&					&\\
of overgroup $K$							&and s.t.			&$N_K(H)$					&$\nu_K(H)$		&$\mu_G(K)$\\ \hline
$R(3^k)$									&--					&$2 \times 3^k$			&$3^{n-k}$			&$\mu(n/k)$\\
$(3^k)^{1+1+1}\colon3^k-1$				&--					&$2 \times 3^k$			&$3^{n-k}$			&$-\mu(n/k)$\\
$2 \times L_2(3^k)$							&$k>1$				&$2 \times 3^k$			&$3^{n-k}$			&$-\mu(n/k)$\\
$2 \times (3^{k}\colon\frac{3^k-1}{2})$		&$k>1$				&$2 \times 3^k$			&$3^{n-k}$			&$\mu(n/k)$\\
$3^k+\sqrt{3^{k+1}}+1\colon6$				&--					&6							&$3^{n-1}$			&$-\mu(n/k)$\\
$3^k-\sqrt{3^{k+1}}+1\colon6$				&$k>1$				&6							&$3^{n-1}$			&$-\mu(n/k)$\\
$(2^2 \times D_{(3^k+1)/2})\colon3$		&$k>1$				&6							&$3^{n-1}$			&$-\mu(n/k)$\\
$2 \times L_2(3)$							&--					&6							&$3^{n-1}$			&$-2\mu(n)$\\ \hline
\end{tabular}\end{center}
The table is organised in such a way that summations which cancel are adjacent. From this it is tedious but not hard to see that $\mu_G(C_6)=0$.

%C_3*
\subsection{$H \cong \langle \T \rangle \in \cent_3^*$}
\begin{center}\begin{tabular}{c c c c c}\hline
Isomorphism type							&for $k \vert n$		&								&					&\\
of overgroup $K$							&and s.t.			&$N_K(H)$						&$\nu_K(H)$		&$\mu_G(K)$\\ \hline
$R(3^k)$									&--					&$3^k \times (3^k\colon2)$		&$3^{2(n-k)}$		&$\mu(n/k)$\\
$(3^k)^{1+1+1}\colon(3^k-1)$				&--					&$3^k \times (3^k\colon2)$		&$3^{2(n-k)}$		&$-\mu(n/k)$\\
$2 \times L_2(3^k)$							&$k>1$				&$2 \times 3^k$				&$3^{2n-k}$		&$-\mu(n/k)$\\
$2 \times (3^{k}\colon\frac{3^k-1}{2})$		&$k>1$				&$2 \times 3^k$				&$3^{2n-k}$		&$\mu(n/k)$\\
$3^k+\sqrt{3^{k+1}}+1\colon6$				&--					&6								&$3^{2n-1}$		&$-\mu(n/k)$\\
$3^k-\sqrt{3^{k+1}}+1\colon6$				&$k>1$				&6								&$3^{2n-1}$		&$-\mu(n/k)$\\
$(2^2 \times D_{(3^k+1)/2})\colon3$		&$k>1$				&6								&$3^{2n-1}$		&$-\mu(n/k)$\\
$2 \times L_2(3)$							&--					&6								&$3^{2n-1}$		&$-2\mu(n)$\\ \hline
\end{tabular}\end{center}
As in the previous case, summations which cancel have been arranged to that they are adjacent and again we see that $\mu_G(C_3^*)=0$.

%C_2
\subsection{$H \cong \langle \J \rangle \in \cent_2$}
\begin{center}\begin{tabular}{c c c c c}\hline
Isomorphism type							&for $k \vert n$		&											&									&\\
of overgroup $K$							&and s.t.			&$N_K(H)$									&$\nu_K(H)$						&$\mu_G(K)$\\ \hline
$R(3^k)$									&--					&$2 \times L_2(3^k)$						&$3^n(3^{2n}-1)/3^k(3^{2k}-1)$		&$\mu(n/k)$\\
$2 \times L_2(3^k)$							&$k>1$				&$2 \times L_2(3^k)$						&$3^n(3^{2n}-1)/3^k(3^{2k}-1)$		&$-\mu(n/k)$\\
$(3^k)^{1+1+1}\colon(3^k-1)$				&--					&$2 \times (3^k\colon\frac{3^k-1}{2})$		&$3^n(3^{2n}-1)/3^k(3^k-1)$		&$-\mu(n/k)$\\
$2 \times (3^{k}\colon\frac{3^k-1}{2})$		&$k>1$				&$2 \times (3^{k}\colon\frac{3^k-1}{2})$		&$3^n(3^{2n}-1)/3^k(3^k-1)$		&$\mu(n/k)$\\
$(2^2 \times D_{(3^k+1)/2})\colon3$		&$k>1$				&$2^2 \times D_{(3^k+1)/2}$				&$3^n(3^{2n}-1)/2(3^k+1)$			&$-\mu(n/k)$\\
$2^2 \times D_{(3^k+1)/2}$					&$k>1$				&(3) $2^2 \times D_{(3^k+1)/2}$			&$3^n(3^{2n}-1)/2(3^k+1)$			&$3\mu(n/k)$\\
$2 \times L_2(3^k)$							&$k>1$				&(2) $2^2 \times D_{(3^k+1)/2}$			&$3^{n}(3^{2n}-1)/(3^k+1)$			&$-\mu(n/k)$\\
$3^k+\sqrt{3^{k+1}}+1\colon6$				&--					&6											&$3^{n-1}(3^{2n}-1)/2$				&$-\mu(n/k)$\\
$3^k-\sqrt{3^{k+1}}+1\colon6$				&$k>1$				&6											&$3^{n-1}(3^{2n}-1)/2$				&$-\mu(n/k)$\\
$(2^2 \times D_{(3^k+1)/2})\colon3$		&$k>1$				&(1) $2^3$, (1) $2\times L_2(3)$			&$3^{n-1}(3^{2n}-1)/2$				&$-\mu(n/k)$\\
$2^2 \times D_{(3^k+1)/2}$					&$k>1$				&(4) $2^3$									&$3^{n-1}(3^{2n}-1)/2$				&$3\mu(n/k)$\\
$2 \times L_2(3)$							&--					&(2) $2^3$, (1) $2 \times L_2(3)$			&$7.3^{n-1}(3^{2n}-1)/8$			&$-2\mu(n)$\\
$2^3$										&--					&(7) $2^3$									&$3^{n-1}(3^{2n}-1)/8$				&$21\mu(n)$\\ \hline
\end{tabular}\end{center}
As in the case $H \cong V$ we organise the table according to conjugacy classes of subgroups isomorphic to $H$ and eventually find that $\mu_G(2)=0$.

%1
\subsection{$H \cong 1 \in \eye$} A conjecture of Conder states that if $G$ is a non-abelian almost simple group, then $\mu_G(1)$ is divisible by the order of a minimal normal subgroup of $G$.
\begin{center}\begin{tabular}{c c c c c}\hline
Isomorphism type							&for $k \vert n$			&												&\\
of overgroup $K$							&and s.t.				&$\nu_K(H)$									&$\mu_G(K)$\\ \hline
$R(3^k)$									&--						&$\vert G \vert/3^{3k}(3^{3k}+1)(3^k-1)$		&$\mu(n/k)$\\
$3^k+\sqrt{3^{k+1}}+1\colon6$				&--						&$\vert G \vert/6(3^k+\sqrt{3^{k+1}}+1)$		&$-\mu(n/k)$\\
$3^k-\sqrt{3^{k+1}}+1\colon6$				&$k>1$					&$\vert G \vert/6(3^k-\sqrt{3^{k+1}}+1)$			&$-\mu(n/k)$\\
$(3^k)^{1+1+1}\colon(3^k-1)$				&--						&$\vert G \vert/3^{3k}(3^k-1)$					&$-\mu(n/k)$\\
$2 \times L_2(3^k)$							&$k>1$					&$\vert G \vert/3^k(3^{2k}-1)$					&$-\mu(n/k)$\\
$2 \times (3^{k}\colon\frac{3^k-1}{2})$		&$k>1$					&$\vert G \vert/3^k(3^k-1)$						&$\mu(n/k)$\\
$(2^2 \times D_{(3^k+1)/2})\colon3$		&$k>1$					&$\vert G \vert/6(3^k+1)$						&$-\mu(n/k)$\\
$2^2 \times D_{(3^k+1)/2}$					&$k>1$					&$\vert G \vert/6(3^k+1)$						&$3\mu(n/k)$\\
$2 \times L_2(3)$							&--						&$\vert G \vert/24$								&$-2\mu(n)$\\
$2^3$										&--						&$\vert G \vert/168$							&$21\mu(n)$\\ \hline
\end{tabular}\end{center}
After some lengthy calculation, we find that $\mu_G(1)=0$, supporting this conjecture.

\subsection{The proof of Theorem \ref{mob}}
From the determination of the \M function for subgroups $H \sg G$ and the Lemmas in Section \ref{cnr} the \M function for all of $G$ can be verified and is as appears in Theorem \ref{mob}, as was to be shown. In the case when $G=R(27)$ the full subgroup lattice and \M function has been determined by Connor and Leemans \cite{conn} and is seen to agree with our determinations, aside from a number of errors in their calculations as of October 2014.

%%%%%%%%%%%%%%%%%%%%%%%%%%%---Applications---%%%%%%%%%%%%%%%%%%%%%%%%%%%
\section{Applications of the \M function} \label{app}

In this section we state a number of results which are corollaries of Theorem \ref{mob}. Throughout $\vert H \vert_n$ denotes the size of the set of elements in the group $H$ having order $n$.

%Free
\subsection{Free products}
Let $F_2$ be the free group on 2 generators and $C_n * C_\infty$, for a positive integer $n$, be the free product with presentation $\langle x,y \mid x^n=1 \rangle$. Our main corollary is then the followin.

\begin{cor} \label{f2} Let $G=R(3^n)$ be a simple small Ree group. The number of inequivalent generating pairs of elements of $G$ is
\[d_2(G)=\frac{\phi(G)}{\vert\textnormal{Aut}(G)\vert}=\frac{1}{n}\sum_{l \vert n}\mu\left(\frac{n}{l}\right)(3^l-1)(3^{6l}-3^{2l}-16).\]\end{cor}

\begin{remark} The quantity $d_2(G)$ has a number of other interpretations, a few of which we mention here. \begin{itemize}
\item if $G$ is simple, this is equal to the largest positive integer, $d$, such that $G^d$ can be 2-generated,
\item in Grothendieck's theory of dessins d'enfants \cite{dd} this is equal to the number of distinct regular dessins with automorphism group isomorphic to $G$,
\item the number of oriented hypermaps having automorphism group isomorphic to $G$ \cite{dj87}.
\end{itemize}\end{remark}

The values of $d_2(R(3^n))$ for the first few values of $n$ are then as follows.
\[\begin{tabular}{c r}\hline
$G$		&$d_2(G)$\\ \hline
$R(3^3)$	&3\,357\,637\,312\\
$R(3^5)$	&9\,965\,130\,790\,521\,984\\
$R(3^7)$	&34\,169\,987\,177\,353\,651\,660\,608\\
$R(3^9)$	&127\,166\,774\,444\,890\,319\,085\,083\,766\,720\\ \hline
\end{tabular}\]

Our remaining examples in this chapter are free products of cyclic groups. In order to prove the next Lemma, we will need the results in Table \ref{hn} which are easily obtained from the character table of $R(q)$ found in \cite{ward}.

\begin{center}\begin{table}\begin{tabular}{c c c c c c c}\hline
Isomorphism							&								&								&\\
type of $H \sg G$						&$\vert H \vert_2$				&$\vert H \vert_3$				&$\vert H \vert_6$\\ \hline
$R(3^h)$								&$3^{2h}(3^{2h}-3^h+1)$		&$(3^{3h}+1)(3^{2h}-1)$		&$3^{2h}(3^{3h}+1)(3^h-1)$\\
$3^h+\sqrt{3^{h+1}}+1\colon6$			&$3^h+\sqrt{3^{h+1}}+1$		&$2(3^h+\sqrt{3^{h+1}}+1)$	&$2(3^h+\sqrt{3^{h+1}}+1)$\\
$3^h-\sqrt{3^{h+1}}+1\colon6$			&$3^h-\sqrt{3^{h+1}}+1$		&$2(3^h-\sqrt{3^{h+1}}+1)$		&$2(3^h-\sqrt{3^{h+1}}+1)$\\
$(3^h)^{1+1+1}\colon3^h-1$			&$3^{2h}$						&$3^{2h}-1$					&$3^{2h}(3^h-1)$\\
$2 \times L_2(3^h)$					&$3^{2h}-3^h+1$				&$3^{2h}-1$					&$3^{2h}-1$\\
$2 \times (3^{h}\colon\frac{3^h-1}{2})$	&1								&$3^h-1$						&$3^h-1$\\
$(2^2 \times D_{(3^h+1)/2})\colon3$	&$3^h+4$						&$2(3^h+1)$					&$2(3^h+1)$\\
$2^2 \times D_{(3^h+1)/2}$				&$3^h+4$						&--								&--\\
$2 \times L_2(3)$						&7								&8								&8\\
$2^3$									&7								&--								&--\\ \hline
\end{tabular} \caption{\label{hn} Values of $\vert H \vert_n$ for $n=2,3$ and 6.} \end{table}\end{center}

\begin{cor} \label{cn} Let $G=R(3^n)$ be a simple small Ree group and $\phi_{n_1,\dots,n_i}(G)=\vert\textnormal{Epi}(C_{n_1}*\dots*C_{n_i},G)\vert$ where $n_i \in \mathbb{N} \cup \{\infty\}$. Then \begin{enumerate}
\item \[\phi_{2,\infty}(G)=\vert G \vert\sum_{l \vert n}\mu\left(\frac{n}{l}\right)(3^l-1)(3^{3l}-3^l-2),\]
\item \[\phi_{2,2,2}(G)=\vert G \vert\sum_{l \vert n}\mu\left(\frac{n}{l}\right)(3^l-1)(3^{4l}-3^{3l}+2.3^{2l}-1)\]%(3^{5h}-2.3^{4h}+3.3^{3h}-2.3^{2h}-3^h+17)
\item \[\phi_{3,\infty}(G)=\vert G \vert\sum_{l \vert n}\mu\left(\frac{n}{l}\right)(3^l-1)(3^{4l}-3^{3l}-3^l-4),\]
\item \[\phi_{3,3}(G)=\vert G \vert\sum_{l \vert n}\mu\left(\frac{n}{l}\right)3^l(3^{2l}+3^l-4),\]%(3^{3h}+3^{2h}-4.3^h-2)
\item \[\phi_{6,\infty}(G)=\vert G \vert\sum_{l \vert n}\mu\left(\frac{n}{l}\right)(3^l-1)(3^{5l}-3^l-6)\] and %CHECK THIS...
\item \[\phi_{9,\infty}(G)=\vert G \vert\sum_{l \vert n}\mu\left(\frac{n}{l}\right)3^{5l}(3^l-1).\]
\end{enumerate}
\end{cor}

\begin{proof}
In each case we set $\sigma_\Gamma(H)=\vert H \vert \vert H \vert_n$ for the appropriate value of $n$. Values of $\vert H \vert_n$ for $n=2,3$ and 6 can be found in Table \ref{hn}, for $n=9$ these elements are found only in subfield subgroups and their parabolic subgroups and from the character table \cite{ward} we have $\vert R(3^h) \vert_9=3^{2h}(3^{3h}+1)(3^h-1)$ and $\vert (3^h)^{1+1+1}\colon3^h-1 \vert_9=3^{2h}(3^h-1)$.
\end{proof}

\begin{remark} The quantities $\phi_{2,\infty}$ and $\phi_{2,2,2}$ are of interest in the study of regular polytopes as they correspond to the number of regular hypermaps and, respectively, orientably regular maps  having automorphism group isomorphic to $G$. We refer the reader to \cite{dj87,suzenum} for more details.
\end{remark}

%Hecke
\subsection{Hecke groups}
The Hecke group, $H_n$, for a natural number $n>2$ is defined as follows \[H_n = \langle x,y \mid x^2=y^n=1 \rangle\] and in the case $n=3$ this is isomorphic to the modular group $PSL_2(\ZZ)$. To determine $\phi_{H_n}=\eta_n$ our summatory function becomes $\sigma_{H_n}(H) = \vert H \vert_2 \vert H \vert_n$. Then \[\eta_n = \vert G \vert\sum_{H \sg G}\mu_G(H)\vert H \vert_2 \vert H \vert_n.\] Since $R(q)$ does not contain elements of orders 4, 5 or 8, we will consider $H_3$, $H_6$, $H_7$ and $H_9$.

\begin{cor} \label{eta} Let $G=R(3^n)$ be a simple small Ree group and $\eta_n=\vert$\textnormal{Epi}$(H_n,G)\vert$. Then \begin{enumerate}
\item \[\eta_3(G)=\vert G \vert\sum_{l \vert n}\mu\left(\frac{n}{l}\right)(3^l-1)^2\]
\item \[\eta_6(G)=\vert G \vert\sum_{l \vert n}\mu\left(\frac{n}{l}\right)3^l(3^{2l}-3^l-2)\]
\item \[\eta_9(G)=\vert G \vert\sum_{l \vert n}\mu\left(\frac{n}{l}\right)3^{2l}(3^l-1).\]
\end{enumerate}
\end{cor}

It is well-know, see for example \cite{rr, malle}, that the simple small Ree groups are quotients of the modular group $PSL_2(\ZZ)$. With the \M function we can be a little more precise.

\begin{cor} Let $G=R(3^n)$ be a simple small Ree group and $\Gamma=PSL_2(\ZZ)$ be the modular group. The number of distinct normal subgroups, $N$, of $\Gamma$ such that $\Gamma/N \cong G$ is given by \[d_\Gamma(G)=\frac{\eta_3(G)}{\vert\textnormal{Aut}(G)\vert}=\frac{1}{n}\sum_{l \vert n}\mu\left(\frac{n}{l}\right)(3^l-1)^2.\]
In particular, $G^d$, is a quotient of the modular group for any integer $1 \leq d \leq d_\Gamma(G)$.
\end{cor}

\begin{remark} Finally, we remark that the \M function can also used to determine the number of Hurwitz triples of $G$, that is generating sets $\langle x,y,z \rangle$ such that $x^2=y^3=z^7=xyz=1$. Groups for which such a generating set occurs are known as Hurwitz groups and their study is well documented, see \cite{cond1, cond2} for Conder's surveys of this area. We shall say no more about them here since it was proven by Malle \cite{malle} and independently by Jones \cite{rr} that the simple small Ree groups are Hurwitz groups.
\end{remark}

%Probabilistic
\subsection{Probabilistic generation} Let $P_{a,b}(G)$ be the probability that the group $G$ can be generated by a randomly chosen element of order $a$ and a randomly chosen element of order $b$, where $a,b \geq 2$ are natural numbers or $\infty$ to mean any randomly chosen element, irrespective of its order.

\begin{proof}[Proof of Corollary \ref{23ree}] From Corollaries \ref{cn} and \ref{eta} and Table \ref{hn} we have
\[P_{2,3}(G)=\frac{\eta_{3}}{\vert G \vert_2 \vert G \vert_3}=\frac{3^n\sum_{l \vert n}\mu\left(\frac{n}{l}\right)(3^l-1)^2}{(3^{3n}+1)}\]%3^{2l}-2.3^l-2
and
\[P_{3,3}(G)=\frac{\phi_{3,3}}{\vert G \vert_3 \vert G \vert_3}=\frac{3^{3n}\sum_{l \vert n}\mu\left(\frac{n}{l}\right)3^l(3^{2l}+3^l-4)}{(3^{3n}+1)(3^{2n}-1)(3^n+1)}\]%3^{2l}-2.3^l-2
both of which tend to 1 as $\vert G \vert \to \infty$.\end{proof}

We are also able to deduce a number of other probabilistic results, all of which are similarly proved.

\begin{cor} Let $G=R(3^n)$ be a small Ree group, $P_{a,b}(G)$ as before and $P_{2,2,2}(G)$ be the probability that three randomly chosen involutions generate $G$. Then \begin{enumerate}
\item $P_{\infty,\infty}(G)$,
\item $P_{2,\infty}(G)$,
\item $P_{3,\infty}(G)$,
\item $P_{6,\infty}(G)$,
\item $P_{9,\infty}(G)$,
\item $P_{2,6}(G)$,
\item $P_{2,9}(G)$ and
\item $P_{2,2,2}(G)$
\end{enumerate}
all tend to 1 as $\vert G \vert \to \infty$. \end{cor}

\begin{remark} The first three results are due to Liebeck and Shalev proves using different methods to those we employ here. The first result appears in \cite{ls95} while the second and third appear in \cite{ls} as Theorems 1.1 and 1.2 respectively. \end{remark}

%acknow
\section*{Acknowledgements}
The author wishes to thank his supervisor Ben Fairbairn for tireless patience and valuable guidance; Gareth Jones for originally suggesting the problem and many further helpful conversations regarding its resolution; and the organisers of SIGMAP 2014 and the conversation with Dimitri Leemans which it afforded. The author would also like to thank Rob Wilson, Peter Cameron and Jeroen Schillewaert for their time and comments regarding the geometry of the small Ree groups.

\small
%biblio


\begin{thebibliography}{99}
\bibliographystyle{plainnat}

\bibitem{cond1} \textsc{M. Conder},
`Hurwitz groups: A brief survey',
\emph{Bull. Amer. Math. Soc.}
23 (1990) 359--370.

\bibitem{cond2} \textsc{M. Conder},
`An update on Hurwitz groups',
\emph{Groups Complex. Cryptol.}
(1) 2 (2010) 35--49.

\bibitem{conn} \textsc{T. Connor {\normalfont and} D. Leemans},
`An atlas of subgroup lattices of finite almost simple groups'
Preprint, 2013, arXiv:math.GR/ arXiv:13064820.

\bibitem{ATLAS} \textsc{J. H. Conway, R. T. Curtis, S. P. Norton, R. A. Parker {\normalfont and} R. A. Wilson},
\emph{Atlas of Finite Groups}
(Clarendon Press, Aynsham, 1985).

\bibitem{dick} \text{L. E. Dickson},
\emph{Linear groups, with an exposition of the Galois field theory}
(Teubner, Leipzig, 1901).

\bibitem{phdowns} \textsc{M. Downs},
`\M inversion of some classical groups and an application to the enumeration of regular maps'
PhD Thesis,
University of Southampton, 1988.

\bibitem{psl2} \textsc{M. Downs},
`The M\"obius function of $PSL_2(q)$, with application to the maximal normal subgroups of the modular group',
\emph{J. Lond. Math. Soc.}
(2) 43 (1991) 61--75.

\bibitem{some} \textsc{M. Downs},
`Some Enumerations of Regular Hypermaps with Automorphism Group isomorphic to $PSL_2(q)$',
\emph{Q. J. Math.}
(1) 48 (1997) 39--58.

\bibitem{dj87} \textsc{M. L. N. Downs \textnormal{and} G. A. Jones},
`Enumerating regular objects with a given automorphism group',
\emph{Discrete Math.}
64 (1987) 299--302.

\bibitem{suzenum} \textsc{M. L. N. Downs \textnormal{and} G. A. Jones},
`Enumerating regular objects associated with Suzuki groups',
Preprint, 2013, arXiv:math.GR/ arXiv:13095215.

\bibitem{suzmob} \textsc{M. L. N. Downs \textnormal{and} G. A. Jones},
`The M\"obius Function of the Suzuki Groups, with Applications to Enumeration',
Preprint, 2014, arXiv:math.GR/ arXiv:14045470.

%\bibitem{gap} The GAP~Group,
%`GAP -- groups, algorithms and programming, version 4.7.4',
%2014, \\http://www.gap-system.org.

\bibitem{dd} \textsc{A. Grothendieck},
`Esquisse d'un programme',
\emph{Geometric Galois Actions I, Around Grothendieck's Esquisse d'un Programme} (eds P. Lochak and L. Schneps), 
London Math. Soc. Lecture Note Ser. 242 (Cambridge University Press, Cambridge, 1997), 5--48.

\bibitem{hall} \textsc{P. Hall},
`The Eulerian functions of a group',
\emph{Q. J. Math.}
(1) 7 (1936) 134--151.

\bibitem{rr} \textsc{G. A. Jones},
`Ree Groups and Riemann Surfaces',
\emph{J. Algebra}
(1) 165 (1994) 41--62.

\bibitem{kleid} \textsc{P. B. Kleidman},
`The Maximal Subgroups of the Chevalley Groups $G_2(q)$ with $q$ Odd, the Ree Groups $^2G_2(q)$, and Their Automorphism Groups',
\emph{J. Algebra}
117 (1988) 30--71.

\bibitem{ln} \textsc{V. M. Levchuk \textnormal{and} Ya. N. Nuzhin},
`Structure of Ree groups', (Russian)
\emph{Algebra Logika}
24 (1985) 26--41;
English translation,
\emph{Algebra Logic}
24 (1985) 16--26.

\bibitem{lieb} \textsc{M. W. Liebeck},
`Probabilistic and asymptotic aspects of finite simple groups',
\emph{Probabilistic Group Theory, Combinatorics, and Computing} (eds A. Detinko, D. Flannery and E. O'Brien),
Lecture Notes in Mathematics Volume 2070 (Springer London, London, 2013) 1--34.

%\bibitem{ls23} \textsc{M. W. Liebeck \textnormal{and} A. Shalev},
%`Classical groups, probabilistic methods, and the $(2,3)$-generation problem'
%\emph{Ann. of Math.}
%(144 (1996) 77--125.

\bibitem{ls95} \textsc{M. W. Liebeck \textnormal{and} A. Shalev},
`The probability of generating a finite simple group',
\emph{Geom. Dedicata}
56 (1995) 103--113.

\bibitem{ls} \textsc{M. W. Liebeck \textnormal{and} A. Shalev},
`Simple Groups, Probabilistic Methods, and a Conjecture of Kantor and Lubotzky',
\emph{J. Algebra}
184 (1996) 31--57.

\bibitem{lun} \textsc{H. L\"uneberg},
`Some Remarks Concerning the Ree Groups of Type $(G_2)$'
\emph{J. Algebra}
3 (1966) 256--259.

\bibitem{malle} \textsc{G. Malle},
`Hurwtiz groups and $G_2(q)$'
\emph{Canad. Math. Bull.}
33 (1990) 349--357.

\bibitem{ree} \textsc{R. Ree},
`A family of simple groups associated with the simple Lie algebra of type $(G_2)$',
\emph{Am. J. Math.}
83 (1961) 432--462.

\bibitem{tits} \textsc{J. Tits},
`Les groupes simples de Suzuki et de Ree',
\emph{S\'eminaire Bourbaki}
6 (1960) 65--82.

\bibitem{van} \textsc{H. Van Maldeghem},
`Generalized Polygons',
Monographs in Mathematics, vol. 93 (Birkh\"auser, 1998).

\bibitem{ward} \textsc{H. N. Ward},
`On Ree's series of simple groups',
\emph{Trans. Amer. Math. Soc.}
121 (1966) 62--89.

\bibitem{raw1} \textsc{R. A. Wilson},
`A new construction of the Ree groups of type $^2G_2$',
\emph{Proc. Edinb. Math. Soc.}
53, (2010) 531--542.

\bibitem{raw2} \textsc{R. A. Wilson},
`Another new approach to the small Ree groups',
\emph{Arch. Math. (Basel)}
94 (2010) 501--510.

\bibitem{raw12} \textsc{R. A. Wilson},
`On the simple groups of Suzuki and Ree',
\emph{Proc. Lond. Math. Soc.}
107 (2013) 680--712.

%\bibitem{raw} \textsc{R. A. Wilson},
%\emph{The finite simple groups}
%Graduate Texts in Mathematics 251 (Springer-Verlag London Ltd., London, 2009).

%\bibitem{brauer} \textsc{R. A. Wilson}, \emph{et al.},
%\textsc{Atlas} of Finite Group Representations v3, http://brauer.maths.qmul.ac.uk/Atlas/v3.

\end{thebibliography}
\end{document}